\documentclass[12pt]{amsart}

\usepackage{graphicx}
\usepackage{hyperref}
\usepackage[usenames,dvipsnames,x11names]{xcolor}
\usepackage[all]{xy}

\newtheorem{theorem}{Theorem}
\newtheorem{lemma}[theorem]{Lemma}
\newtheorem{proposition}[theorem]{Proposition}
\newtheorem{corollary}[theorem]{Corollary}
\newtheorem*{theorem*}{Theorem}

\theoremstyle{definition}

\theoremstyle{remark}
\newtheorem{remark}[theorem]{Remark}
\newtheorem*{remark*}{Remark}
\newtheorem*{proposition*}{Proposition}

\newcommand{\gp}[1]{{\left\langle #1 \right\rangle}}

\let\al\alpha

\let\de\delta

\let\si\sigma
\let\om\omega
\let\Ga\Gamma
\let\De\Delta
\let\La\Lambda

\def\ovc{{\overline{c}}}

\def\ovu{{\bar{u}}}
\def\ovv{{\bar{v}}}
\def\ovw{{\bar{w}}}
\def\ovX{{\bar{X}}}
\def\CC{{\mathcal C}}

\def\CG{{\mathcal G}}
\def\CL{{\mathcal L}}
\def\CS{{\mathcal S}}

\def\NP{{\mathbf{NP}}}

\def\MZ{{\mathbb{Z}}}
\def\MQ{{\mathbb{Q}}}

\def\any{{\quad\forall\,}}

\let\isom\simeq

\let\bd\partial
\let\ti\tilde

\let\seq\subseteq
\let\eset\emptyset

\let\ov\overline

\DeclareMathOperator\supp{supp}
\DeclareMathOperator\diam{diam}

\DeclareMathOperator\Vol{Vol}
\DeclareMathOperator\Sp{Sp}
\DeclareMathOperator\proj{proj}

\DeclareMathOperator\DP{DP}

\newcommand\Ints{{\mathbb Z}}
\newcommand\Reals{{\mathbb R}}
\newcommand\set[1]{\{ #1 \}}
\newcommand\bset[1]{\left\{ #1 \right\}}
\newcommand\setof[2]{\{ #1 \mid #2\}}
\newcommand\sgp[1]{\langle #1 \rangle}

\title[Orientable quadratic equations]{Orientable quadratic equations in free metabelian groups}

\author[]{Igor Lysenok}
\address{Stevens Institute of Technology, Hoboken, NJ 07030, USA
\and
Steklov Institute of Mathematics, Gubkina str.\ 8, 119991 Moscow, Russia} \email{igor.lysenok@gmail.com}
\thanks{The first author has been partially supported by the Russian Foundation for Basic Research}

\author[]{Alexander Ushakov}
\address{Stevens Institute of Technology, Hoboken, NJ 07030, USA} \email{aushakov@stevens.edu}


\date{}


\date{}

\begin{document}
\maketitle

\begin{abstract}
We prove that the Diophantine problem for orientable quadratic equations in free metabelian groups is decidable and furthermore, $\NP$-complete. In the case when the number of variables in the equation is bounded, the problem is decidable in polynomial time.
%
\end{abstract}

\section{Introduction}

Let $G$ be a group, $X$ a set of variables, 
and $F_X$ the free group on $X$.
An {\em equation} in ~$G$ is a formal equality 
$W=1$ where $W \in G \ast F_X$.
A {\em solution} of an equation $W=1$ is a homomorphism 
$\al: G*F_X \to G$ such that $\al(W)=1$
and $\al(g)=g$ for all $g \in G$. 

\medskip\noindent
\textbf{Diophantine problem $\DP$ in a group $G$} for a class of equations $\CC$
is an algorithmic question to decide if a given equation $W=1$ in $\CC$
has a solution, or not.

\subsection{Quadratic equations in a group $G$}
A word $W \in G*F_X$ and an equation $W=1$ in $G$ are {\em quadratic} 
if every variable $x\in X$ occurring in $W$
occurs exactly twice (as $x$ or $x^{-1}$).
The Diophantine problem for quadratic equations in $G$
is quite general, for example it naturally contains the word problem
and the conjugacy problem for $G$. 
The problem received a lot of attention recently 
and was investigated for several classes of groups.
In particular, it was shown in \cite{KLMT} that solving quadratic equations
is $\NP$-complete for free non-abelian groups.
Later the same result was generalized 
in \cite{Kharlampovich-Mohajeri-Taam-Vdovina:2017} 
to the class of non-cyclic hyperbolic torsion-free groups.
Quadratic equations are solvable in the first Grigorchuk group  \cite{Lysenok_Miasnikov_Ushakov:2016}.
Recently Roman'kov showed in \cite{Roman'kov_2016} 
that quadratic equations
are unsolvable in a nilpotent group of class 2.

In general, quadratic equations are more approachable than general ones.
Typically, they can be treated by a specialized method.
For instance, solving quadratic equations in free groups is relatively easy, 
it essentially reduces to enumeration of the minimal automorphic orbit
of the equation (see \cite{CE}), while 
solving general equations require very sophisticated 
tools \cite{Mak,Diekert-Jez-Plandowski:2016}.
Similarly, the main result of this paper demonstrates 
the difference between quadratic
and general equations in free metabelian groups. 
The former are proved to be solvable in $\NP$, while
the latter are known to be unsolvable.

A quadratic equation $W=1$ in $G$ is called {\em orientable}
if each variable $x$ occurring in $W$ occurs once as $x$ 
and once as $x^{-1}$, and {\em non-orientable} otherwise.
Two special forms of orientable and non-orientable equations are
called \emph{standard quadratic equations}:
\begin{equation}\label{eq:orientable-quadratic}
  [x_1,y_1] [x_2,y_2] \dots [x_g,y_g] \, 
  c_1 z_2^{-1} c_2 z_2 \dots z_{m}^{-1} c_{m} z_{m} = 1
    \quad (g \ge0, \ m\ge 0),
\end{equation}
\begin{equation}
    \label{eq:nonorientable-quadratic}
  x_1^2 x_2^2 \dots x_g^2 \, 
  c_1 z_2^{-1} c_2 z_2 \dots z_{m}^{-1} c_{m} z_{m} = 1
    \quad (g >0, \ m\ge 0),
\end{equation}
where $x_i,y_i,z_i$ are variables and $c_i \in G$ are coefficients.
We assume that the coefficient part 
$c_1 z_2^{-1} c_2 z_2 \dots z_{m}^{-1} c_{m} z_{m}$ is void in the case $m=0$.
The choice of terminology comes the following well-known fact
(whose proof essentially repeats the proof of the
classification theorem for compact surfaces with boundary, 
see for example a classical topology 
textbook \cite[Sections 38--40]{Seifert-Threllfal}; 
a detailed argument can be found for example in \cite{GK}).
See also \cite[Theorem 3.2]{CE}.

\begin{proposition}
\label{pr:standard-quadratic}
Let $W \in G * F_X$ be a quadratic word in $G$. 
If $W$ is orientable (non-orientable resp.), 
then there exists an automorphism $\phi$ of
$G * F_X$ with $\phi|_G = \mathit{id}$ such that $\phi(W)$ is of 
the form \eqref{eq:orientable-quadratic} 
(form \eqref{eq:nonorientable-quadratic} resp.). 
Furthermore, the automorphism $\phi$ (written as a composition of some
elementary automorphisms) and the word $\phi(W)$ can be 
computed in time bounded by $O(|W|^2)$.
\end{proposition}

Proposition \ref{pr:standard-quadratic} allows us to restrict our attention
to standard quadratic equations.
In what follows, we prefer to consider standard orientable 
quadratic equations written in a more symmetric form with an extra variable ~$z_1$:
\begin{equation} \label{eq:equation}
  [x_1, y_1] [x_2, y_2] \dots [x_g, y_g] = z_1 c_1 z_1^{-1} z_2 c_2 z_2^{-1} \ldots z_m c_mz_m^{-1}
\end{equation}
which is obviously equivalent to \eqref{eq:orientable-quadratic}.

\subsection{Main results}
In this paper we consider orientable
quadratic equations in the free metabelian group $M_n$ of rank $n\ge 2$.

\begin{theorem}\label{th:main_theorem}
The Diophantine problem for orientable quadratic equations in $M_n$ is solvable.
\end{theorem}

\begin{theorem}\label{th:main_theorem2}
The Diophantine problem for orientable quadratic equations 
in $M_n$ is $\NP$-complete.
\end{theorem}

\begin{theorem}\label{th:main_theorem3}
For fixed $g$ and $m$, the Diophantine problem 
for orientable quadratic equations 
in $M_n$ can be solved by a polynomial time algorithm.
\end{theorem}

In fact, the proof of Theorem \ref{th:main_theorem} in Section \ref{se:complexity}
yields a slightly stronger result.
It works for the uniform version of the Diophantine problem 
in free metabelian groups.

\medskip\noindent
\textbf{Uniform Diophantine problem} for a class of groups $\CG$ 
and a class of equations $\CC$.
Given a group $G\in \CG$ and an equation $W=1$ from $\CC$ with coefficients
from $G$ decide if $W=1$ has a solution, or not.

\begin{theorem}\label{th:main_theorem4}
The uniform Diophantine problem 
for orientable quadratic equations in free metabelian
groups is solvable and, furthermore, $\NP$-complete.
\end{theorem}

\subsection{Known facts on equations in free metabelian groups}
In general, for every $n\ge 2$ the Diophantine problem for
equations of the form $W(x_1,\ldots,x_k)=c$ in $M_n$ 
is undecidable.
By \cite{Roman'kov_1979}, there exist a word $W=W(x_1,\ldots,x_k)$ and an element 
$u\in M_n$ such that the set of integers 
$\{l\in\MZ \mid W=u^l \text{ has a solution}\}$ 
is non-recursive.
In particular, this implies undecidability of the general 
Diophantine problem for $M_n$.
In \cite{Lysenok_Ushakov:2016} the authors proved that
the problem is solvable and, moreover, is $\NP$-complete
for the special class of quadratic equations in $M_n$ of the 
form \eqref{eq:orientable-quadratic} with $g=0$.

It is known that metabelian groups have finite commutator width.
This fact was first established by Rhemtulla in \cite{Rhemtulla-1969} 
who proved that every finitely generated solvable of class $\le3$ 
group has a finite commutator width. 
%
The precise value of the commutator width of $M_n$ is ~$n$.
This is a consequence of a more general result by 
Akhavan-Malayeri and Rhemtulla \cite{MalRhe-1998} and for 
$n \ge 3$ follows also
from several earlier results, see section 7.1 in \cite{Roman'kov:2012}.
In connection with our Theorem \ref{th:main_theorem} this means that 
for an equation of the form \eqref{eq:orientable-quadratic} in $M_n$ 
we can always assume that $g \le n$.

\section{Preliminaries}

\subsection{Elements of $M_n$ as 1-cycles}
We fix a set $\set{a_1,a_2,\dots,a_n}$ of free generators for $M_n$.
By $A_n$ we denote the free abelian group $M_n / M_n'$ and
by $x \mapsto \bar x$ the canonical epimorphism $M_n \to A_n $.
In particular, $\set{\bar a_1, \dots, \bar a_n}$ is a basis for $A_n$. 
We use additive notation for $A_n$.
The Cayley graph of $A_n$ with respect to the generating set $\set{\bar a_1, \dots, \bar a_n}$ is denoted $\Ga_n$.
We identify vertices of $\Ga_n$ with elements of $A_n$ and
assume that directed edges of $\Ga_n$ are labeled 
with letters $a_i^{\pm1}$.

A word $w \in F_n$ determines a unique edge path $p_w$ in $\Ga_n$ labeled by ~$w$
which starts at ~$1$ (the vertex corresponding to the identity of ~$A_n$). It defines a $1$-chain
$\si(w)$ in $\Ga_n$ which is the algebraic sum of all edges traversed by $p_w$; the sum of two mutually inverse directed edges is defined to be ~0.
It is well known that the mapping $w \mapsto \si(w)$ induces a well-defined injective map of $M_n$ to 
the group $C_1(\Ga_n)$ of 1-chains of $\Ga_n$ over $\Ints$;
that is, two words $u$ and $w$ define the same element of $M_n$ if and only $\si(u) = \si(w)$
(see \cite{DLS,Vershik_Dobrynin:2004,Miasnikov_Romankov_Ushakov_Vershik:2010}).
For $g \in M_n$,
we use the same notation $\si(g)$ for the image of $g$ under the induced map.
It is an easy exercise to check that for any $g,h \in M_n$,
\begin{equation} \label{eq:sigma-of-product}
  \si(gh) = \si(g) + \bar g \si(h),
\end{equation}
where $\bar g \si(h)$ is obtained by shifting $\si(h)$ by $\bar g$ (via the action of $A_n$ on $\Ga_n$).
In particular, we have $\si(g^{-1}) = - \bar g \si(g)$ and 
the action of $A_n$ on $C_1(\Ga_n)$ agrees with
conjugation in $M_n$:
$$
    \si(g h g^{-1}) = \bar g \si(h).
$$
In additive notation for $A_n$, we have
$$
    (\bar{g}_1 + \bar{g}_2) \si(h) = \bar{g}_1 (\bar{g}_2 \si(h)),
$$
and, obviously, for the boundary $\bd_1 \si(g)$ we have
$$
    \bd_1 \si(g) = \bar g - 1.
$$
This implies that $g \in M_n'$ if and only if $\si(g)$ is a 1-cycle.
Since $\si(gh) = \si(g) + \si(h)$ if $g \in M_n'$, $\si$ 
induces an isomorphism between $M_n'$ and the group $Z_1(\Ga_n)$
of $1$-cycles of $\Ga_n$.

If $L$ is a subgroup of $A_n$, then by 
$\tau_L(w)$ we denote the projection of $\si(w)$ in 
the quotient $\Ga_n/L$:
$$
  \tau_L: M_n \overset\si\to C_1(\Ga_n) \to C_1(\Ga_n/L).
$$
Now, with a collection of elements $h_1,\dots,h_r \in M_n$ we
associate the subgroup $L = \gp{\bar h_1,\dots, \bar h_r}$ of $A_n$
and two subgroups of $M_n$:
\begin{align*}
H_L' &= \gp{[[a_i, a_j]^g, h_k], \mbox{ where } 1\le i,j\le n,\ 1\le k\le r,\ g \in M_n},\\
H_L & = 
\left\langle
\begin{array}{l}
[[a_i, a_j]^g, h_k], \\
{}[h_i, h_j],
\end{array}
\begin{array}{l}
\mbox{ where }1\le i,j\le n,\ 1\le k\le r,\ g \in M_n \\
\mbox{ where }1\le i,j\le r
\end{array}
\right\rangle.
\end{align*}
The following proposition describes elements of $M_n$ vanishing under ~$\tau_L$.

\begin{proposition}[See {{\cite[Proposition 2.2]{Lysenok_Ushakov:2016}}}] \label{prop:trace-kernel}
$\ker(\tau_L) = H_L$.
\qed
\end{proposition}

{\em Remark.} The technique of representing
elements in $M_n$ as $1$-chains in the Cayley graph of $\MZ^n$
is not new and was introduced several times under different names,
see for example \cite{Miasnikov_Romankov_Ushakov_Vershik:2010, Vershik_Dobrynin:2004}. Equivalently, we could use the language
of Fox derivatives: it can be easily checked that for any word
$w$ representing an element $g \in M_n$
$$
    \si(g) = \sum_i \frac{\partial w}{\partial a_i} \si(a_i)
$$
where $\frac{\partial w}{\partial a}$ is the Fox derivative of $w$
with respect to a generator $a$ of ~$M_n$ and  
images $\si(a_i)$ of generators $a_i$ form
a generating set for $C_1(\Ga_n)$ as an $\Ints M_n$-module.


\subsection{External square}
As a convenient tool, we use the following well known construction.
Let $X$ be an $R$-module over a commutative ring $R$ with identity. 
The {\em external square} $\La^2(X)$
of $X$ defined as the skew-symmetric quotient of the tensor square $X \otimes X$, i.e.\ the quotient over the submodule 
generated by all elements $x \otimes y + y \otimes x$ 
and $x \otimes x$. 
By construction, the map $(x,y) \mapsto x \otimes y$
induces the skew-symmetric bilinear map $X \times X \to \La^2(X)$, 
the wedge product $x \wedge y$.
If $X$ is a free $R$-module with basis $B = \set{b_1,\dots,b_k}$ then $\La^2(X)$ is a free $R$-module with basis
$$
    B^{\wedge2} = \setof{b_i \wedge b_j}{1 \le i < j \le k}.
$$

We consider $\La^2 (A_n)$ as the external square of the $\MZ$-module $A_n$. 
It is a free abelian group with basis $\setof{\bar a_i \wedge \bar a_j}{1 \le i < j \le n}$
and hence is isomorphic to the second factor $M_n' / [M_n',M_n]$ of the upper central series of $M_n$
where $\bar a_i \wedge \bar a_j$ maps to the image of the basis commutator $[a_i,a_j]$.
Thus, we have an epimorphism $\phi: M_n' \to \La^2(A_n)$ given by 
$$
  \phi([a_i,a_j]^g) = \bar a_i \wedge \bar a_j \text{ for any } g \in M_n.
$$
Since metabelian groups satisfy the identities
\begin{equation} \label{eq:comm-linearity}
  [u v,x] = [u,x][v,x], \quad [x, u v] = [x,u][x,v],
\end{equation}
together with linearity of the wedge product we obtain
$$
  \phi([g,h]) = \bar g \wedge \bar h.
$$
If $\set{b_1,b_2,\dots,b_n}$ is a basis for $A_n$
and $u,v \in A_n$ are expressed as
$$
    u = \sum_{i=1}^n t_i b_i \quad\text{and}\quad 
    v = \sum_{i=1}^n s_i b_i,
$$ 
then the definition implies
\begin{equation} \label{eq:wedge-coord}
    u \wedge v = \sum_{1 \le i < j \le n} (t_i s_j - t_j s_i) 
    \, b_i \wedge b_j.
\end{equation}


\begin{lemma} \label{lm:kernels-intersection}
In the notation of Proposition \ref{prop:trace-kernel}, 
$\ker(\phi) \cap \ker(\tau_L) = H_L'$.
\end{lemma}

\begin{proof}
The definition of $\phi$ implies $H_L' \seq \ker\phi$.
Due to identities \eqref{eq:comm-linearity} we may assume without loss of generality that 
$\set{\bar h_1,\dots, \bar h_r}$ is a basis for $L$. 
Then $H_L$ is generated by $H_L'$ and commutators $[h_i,h_j]$, $1 \le i < j \le r$.
Now observe that the images 
$\setof{\bar h_i \wedge \bar h_j}{1 \le i < j \le r}$ of these commutators in $\La^2(A_n)$ are a basis for a free 
abelian subgroup 
(to see this, for example, we can enlarge the basis $\set{\bar h_i}$ to obtain a basis for $\MQ^n \simeq A_n \otimes \MQ$;
then $\set{\bar h_i \wedge \bar h_j}$ is a part of a basis for $\La^2(\MQ^n)$).
This implies that $H_L \cap \ker\phi$ contains no elements outside $H_L'$.
\end{proof}

\section{Abelian reduction}

Let $Y = (u_1, v_1, \dots, u_g, v_g, w_1, \dots, w_m)$ be a tuple of elements of $M_n$.
In this section we give a necessary and sufficient condition for $Y$ to be a solution of the equation \eqref{eq:equation}, i.e., to satisfy
$$
[u_1, v_1] \dots [u_g, v_g] = w_1 c_1 w_1^{-1} \ldots w_m c_m w_m^{-1}.
$$
The condition is formulated purely in terms of the image of $Y$ in $A_n$.

Assume that $Y$ is a solution of \eqref{eq:equation}. 
Note that the left-hand side of \eqref{eq:equation}
belongs to $M_n'$ and, hence, $c_1 c_2 \dots c_m \in M_n'$.
Define a subgroup $L$ in $A_n$:
\begin{equation} \label{eq:L-def}
  L = \sgp{\bar u_1, \bar v_1, \bar u_2, \bar v_2, \dots, \bar u_g, \bar v_g, \bar c_1, \bar c_2, \dots, \bar c_m}.
\end{equation}
Applying $\tau_L$ to the both sides of \eqref{eq:equation} we get
$$
  0 = \tau_L (w_1^{-1} c_1 w_1 \dots w_m^{-1} c_m w_m) =  \sum_{i=1}^m \bar w_i \tau_L(c_i).
$$
Applying $\phi$ to the left-hand side 
and to the right-hand side of \eqref{eq:equation} we get
$$
  \phi([u_1, v_1] [u_2,v_2] \dots [u_g,v_g]) = \bar u_1 \wedge \bar v_1 + \dots + \bar u_g \wedge \bar v_g
$$ 
and
\begin{align*}
  \phi (w_1^{-1} c_1 w_1 \dots w_m^{-1} c_m w_m) &= \phi ([w_1, c_1^{-1}] \dots [w_m, c_m^{-1}]) 
      + \phi(c_1 c_2 \dots c_m) \\
    &= \bar c_1 \wedge \bar w_1 + \dots + \bar c_m \wedge \bar w_m + \phi(c_1 c_2 \dots c_m).
\end{align*}
That gives a necessary condition for $Y$ to be a solution of \eqref{eq:equation}.
Below we prove that it is also sufficient.

\begin{proposition} \label{pr:solution-reduction}
Let 
$$
    \bar Y = (\bar u_1, \bar v_1, \bar u_2, \bar v_2, \dots, \bar u_g, \bar v_g, \bar w_1, \bar w_2, \dots, \bar w_m)
$$ 
be a tuple of elements of ~$A_n$. 
Let $L$ be the subgroup of $A_n$ 
generated by all $\bar u_i$, $\bar v_i$ and $\bar c_i$.
Then $\bar Y$ can be lifted to a solution of the equation \eqref{eq:equation} 
if and only if the following conditions are satisfied:
\begin{align}
  &\sum_{i=1}^m \bar w_i \tau_L(c_i) = 0, 
		  \label{eq:coefficient-condition} \\
  &\sum_{i=1}^g \bar u_i \wedge \bar v_i =  
    \sum_{i=1}^m \bar c_i \wedge \bar w_i + \phi(c_1 c_2 \dots c_m).
		  \label{eq:second-homology-condition}
\end{align}
\end{proposition}

\begin{proof}
We need to prove the `if' part. 
Assume that $\bar{Y}$ satisfies the conditions \eqref{eq:coefficient-condition} and \eqref{eq:second-homology-condition}. Let 
$$
  Y = (u_1, v_1, u_2, v_2, \dots, u_g, v_g, w_1, w_2, \dots, w_m)
$$
be any lift of $\bar Y$  in $M_n$. Define
$$
 W  = [x_1, y_1] \dots [x_g, y_g] z_1 c_1^{-1} z_1^{-1} \ldots z_m c_m^{-1} z_m^{-1},
$$
and associate with $Y$ an element
$$
 W(Y)  = [u_1, v_1] \dots [u_g, v_g] w_1 c_1^{-1} w_1^{-1} \ldots w_m c_m^{-1} w_m^{-1},
$$
obtained by substituting the elements of $Y$ into $W$.
Conditions \eqref{eq:coefficient-condition} and \eqref{eq:second-homology-condition}
with Lemma \ref{lm:kernels-intersection} imply that
$$
W(Y) \in \ker(\phi) \cap \ker(\tau_L) = H_L'.
$$ 
Now, it is sufficient to prove that the set
$$
\{W(Y) \mid Y \mbox{ is a lift of } \bar{Y}\}
$$
is a coset of $H_L'$.

Recall that $H_L'$ is generated by the elements of the 
form $[[a_i,a_j]^g,x]$ where $g \in M_n$ and $x$ belongs 
to the set $\set{u_i,v_j,c_k}$. 
Fix an occurrence of $z_k c_k^{-1} z_k^{-1}$ in $W$:
$$
    W = U z_k c_k^{-1} z_k^{-1} V.
$$
Multiplication $w_k \mapsto w_k [a_i,a_j]^g$ of $w_k$
by a generator of $M_n'$ results in multiplication:
$$
    W(Y) \mapsto [[a_i,a_j]^{g U(Y)w_k},c_k] \,W(Y),
$$ 
of the value of $W$. Since $[a_i,a_j]^g$ depends only on $\bar g$,
the factor 
$$
    [[a_i,a_j]^{g U(X)w_k},c_k]
$$ 
runs over all generators
of $H_L'$ of the form $[[a_i,a_j]^g,c_k]$. In a similar way,
multiplication of $u_k$ and $v_k$ by generators of $M_n'$
produces generators of ~$H_L'$ of the form $[[a_i,a_j]^g,v_k]$
and $[[a_i,a_j]^g,u_k]$ respectively. This finishes the proof.
\end{proof}

We now examine conditions \eqref{eq:coefficient-condition} and \eqref{eq:second-homology-condition}.
Observe that  \eqref{eq:coefficient-condition} depends only on images of the elements $\bar w_i$ in $A_n/L$.
Changing $\bar w_i$'s while keeping their images in $A_n/L$ fixed adds 
to the right-hand side of \eqref{eq:second-homology-condition}
an element of the subgroup $K$ of $\La^2 (A_n)$
defined by
$$
  K = \sgp{\bar c_i \wedge g \text{ for all $i$ and } g \in L}.
$$
Furthermore, for any $k\in K$ we can
find appropriate $\bar w_i$'s that add $k$ to 
the right-hand side of \eqref{eq:second-homology-condition}.
Hence, \eqref{eq:second-homology-condition} 
can be replaced with the following condition:
\begin{equation} \label{eq:reduced-homology-condition}
\sum_{i=1}^g \bar u_i \wedge \bar v_i =  
\sum_{i=1}^m    \bar c_i \wedge \bar w_i + \phi(c_1 c_2 \dots c_m) \bmod K. 
\end{equation}
Note that $K$ is the kernel of the natural 
epimorphism $\La^2 (L) \to \La^2 (L/Q)$, where
$Q = \sgp{\bar c_1, \bar c_2, \dots, \bar c_m}$.

\subsection{Restricting to finitely many $w_i$'s}

Here we provide an effective bound on 
the size of elements $\bar w_i$.

\begin{proposition} \label{pr:bounding-w}
An equation \eqref{eq:equation} has a solution if and only if there exists a tuple $\bar Y$ of elements $\bar u_i$, $\bar v_i$ 
and $\bar w_i$ satisfying \eqref{eq:coefficient-condition},
\eqref{eq:reduced-homology-condition} and
\begin{equation}\label{eq:w_bound}
|\bar w_i| \le \sum_{j=1}^m |c_j|, \quad i=1,2,\ldots,m.
\end{equation}
\end{proposition}

\begin{proof}
We need to prove the `only if' part.
Let $\bar Y$ be a tuple satisfying \eqref{eq:coefficient-condition} and \eqref{eq:reduced-homology-condition}.
Below we prove that $\bar Y$ can be modified to satisfy
\eqref{eq:coefficient-condition}, \eqref{eq:reduced-homology-condition}, and
\eqref{eq:w_bound}.
The argument essentially repeats the proof of 
\cite[Proposition 2.6]{Lysenok_Ushakov:2016}.

Denote $I = \set{1,\dots,m}$ and for each $i=1,\dots,m$ define
$$
  \tau_i = \bar w_i \tau_L(c_i).
$$
It follows from \eqref{eq:coefficient-condition} that
$\sum_{i=1}^m \tau_i = 0.$
Given a 1-chain $\rho \in C_1(\Ga_n/L)$,
define $\supp (\rho)$ to be the set of edges in $\Ga_n/L$
that occur in $\rho$ with a non-zero coefficient.
We call a non-empty subset $J \seq I$ a {\em cluster} if $\supp(\tau_i) \cap \supp(\tau_j) = \eset$
for any $i\in J$ and $j \notin J$. It follows from the definition that if $J_1$ and $J_2$ are clusters
and $J_1 \cap J_2\ne\emptyset$, then $J_1 \cap J_2$ is a cluster. 
Hence, $I$ can be partitioned into a
finite disjoint union of minimal clusters.

We introduce an integer-valued distance function $d(e,f)$ on the set $E(\Ga_n/L)$ of edges of $\Ga_n/L$.
By definition, the distance between edges $e$
and $f$ is the distance between their midpoints in the graph $\Ga_n/L$ (where all edges are assumed to have
length $1$). For example, $d(e,f)=1$ if and only if $e$ and $f$ are distinct and have a common  vertex.
The following statements follow directly from the definition of a cluster:
\begin{enumerate}
\item
If $J$ is a cluster, then $\sum_{i \in J} \tau_i = 0$.
\item
If $J$ is a minimal cluster, then:
$$
  \diam\left( \bigcup_{i \in J} \supp(\tau_i) \right) \le \sum_{i\in J} |c_i|.
$$
\end{enumerate}

Note that the composition 
$M_n \overset{\tau_L}\to C_1(\Ga_n/L) \to C_1(\Ga_n/A_n) \isom A_n$ gives the canonical epimorphism $g \mapsto \bar g$. Hence, (i) implies that
$$
    \ov {\prod_{i\in J} w_{i} c_{i} w_{i}^{-1}}
    = \sum_{i\in J} \bar{c}_{i} = 0.
$$

Now consider an arbitrary minimal cluster 
$J = \set{j_1,j_2,\dots,j_k}$.
By~(ii), there exist $r_2,r_3,\dots,r_k \in L$ such that
$$
    |\bar{w}_{j_t} - \bar{w}_{j_1} - r_t| \le \sum_{i\in J} |c_i|,
    \quad t=2,3,\dots,k.
$$
Define new values of $\bar{w}_i$ for $i \in J$ by
$$
    \bar{w}_{j_1}^* = 0, \quad
    \bar{w}_{j_t}^* = \bar{w}_{j_t} - \bar{w}_{j_1} - r_t, 
    \ t=2,3,\dots,k.
$$
Then
$$
    \sum_{i \in J} \bar{w}_i^* \tau_L(c_i) = 
    \bar{w}_{j_1}^{-1} \sum_{i \in J} \tau_i = 0 
$$
and
\begin{align*}
    \sum_{i \in J} \bar{c_i} \wedge \bar{w}_i^* &= 
    \sum_{i \in J} \bar{c_i} \wedge \bar{w}_i 
- \left( \sum_{i \in J} \bar{c_i} \right) \wedge \bar{w}_{j_1} 
    \pmod K \\
    &= \sum_{i \in J} \bar{c_i} \wedge \bar{w}_i . 
\end{align*}
Therefore, replacing $\ovw_i$ with $\ovw_i^\ast$  
in $\ovX$ for each $i\in J$ we 
preserve conditions \eqref{eq:coefficient-condition}
and \eqref{eq:reduced-homology-condition}.
Performing the procedure for each minimal cluster ~$J$ we achieve
the required bound 
$|\bar w_i^*| \le \sum_{i=1}^m |c_i|$ for each $i\in I$.
\end{proof}

Enumerating (finitely many) values of $\bar w_i$'s 
satisfying condition \eqref{eq:w_bound} 
we eliminate $\bar w_i$'s from 
\eqref{eq:coefficient-condition} and \eqref{eq:reduced-homology-condition},
and further reduce the equation \eqref{eq:equation}. 

\begin{corollary} \label{co:constants-reduction}
An equation \eqref{eq:equation} can be effectively reduced to
a disjunction of finitely many systems of the form
\begin{equation} \label{eq:constants-reduction}
\begin{cases}
  \sgp{\bar u_1, \bar v_1, \bar u_2, \bar v_2, \dots, \bar u_g, \bar v_g, 
    \bar c_1, \bar c_2, \dots, \bar c_m} = L, \\
  \pi_L(\de) = 0, \\
  \bar u_1 \wedge \bar v_1 + \dots + \bar u_g \wedge \bar v_g =  h \bmod K,
\end{cases}
\end{equation}
with constant entries $\bar c_1, \bar c_2, \dots, \bar c_m \in A_n$,
$\de \in C_1(\Ga_n)$, $h \in \La^2(A_n)$ and 
unknowns $\bar u_i$, $\bar v_i \in A_n$ and $L \le A_n$.
\end{corollary}

\begin{proof}
By Proposition \ref{pr:bounding-w}, equation \eqref{eq:equation}
is effectively reduced to the disjunction of
systems \eqref{eq:coefficient-condition} $\&$ \eqref{eq:reduced-homology-condition} for finitely 
many values of $\bar w_i$'s.
Observe that the left-hand side of \eqref{eq:coefficient-condition}
can be rewritten as $\tau_L(d)$, where 
$$
    d=w_1^{-1} c_1 w_1 \dots w_m^{-1} c_m w_m.
$$
By definition, $\tau_L(d) = \pi_L(\si(d))$,
where $\pi_L: C_1(\Ga_n) \to C_1(\Ga_n/L)$ is
the quotient map. The value of $\si(d)$
depends only on $\bar w_i$ and $c_i$ and can be expressed as
$$
    \si(d) = \sum_i \bar w_i \si(c_i).
$$
Then equation \eqref{eq:coefficient-condition} can be written as
$\pi_L(\de) = 0$, where $\de = \si(d)$, and 
equation \eqref{eq:reduced-homology-condition} can be written as
$$
    \bar u_1 \wedge \bar v_1 + \dots + \bar u_g \wedge \bar v_g =  h \bmod K,
$$
where 
$
    h=\phi(c_1 c_2 \dots c_m) + \sum_{i=1}^m \bar c_i \wedge \bar w_i.
$
\end{proof}


\subsection{Symplectic transformations}

Recall that $\bar u_i$ and $\bar v_i$ in \eqref{eq:constants-reduction} are images in $A_n$ of the variables 
in the commutator part of the initial equation \eqref{eq:equation}. The stabilizer of the conjugacy class
of the product of commutators in the free group, modulo inner automorphisms, is isomorphic to the mapping
class group of the closed orientable surface of genus ~$g$. 
This group acts on the set of abelianized 
tuples $(\bar u_1, \bar v_1, \dots, \bar u_g, \bar v_g)$
as the symplectic group $\Sp(2g,\MZ)$,
see for example \cite[Theorem 6.4]{FM}.
In particular, the set of 
tuples $\set{\ti u_i, \ti v_i}$ satisfying \eqref{eq:constants-reduction} is invariant under symplectic
transformations. It is well known that the complete set of symplectic transformations over $\MZ$ is generated by the
following transformations:
\begin{enumerate}
\item[(S1)] 
Transposition:
$
  (\bar u_i \to \bar u_j, \ \bar v_i \to \bar v_j, \ \bar u_j \to \bar u_i, \ \bar v_j \to \bar v_i).
$
\item[(S2)]
$SL_2(\Ints)$ on a pair:
$
  (\bar u_i \to \bar u_i + t \bar v_i)$ and $(\bar v_i \to \bar v_i + t \bar u_i).
$
\item[(S3)]
Mixing pairs:
$
  (\bar u_i \to \bar u_i + t \bar u_j, \ \bar v_j \to \bar v_j + t \bar v_i), \ i\ne j.
$
\item[(S4)]
Mixing pairs II:
$
  (\bar u_i \to \bar u_i + t \bar v_j, \ \bar u_j \to \bar u_j + t \bar v_i), \ i\ne j,
$\\
(deducible from (S1), (S2), and (S3)).
\end{enumerate}

\begin{remark*}
Formally speaking, we only need the fact that transformations (S1)--(S4)
preserve the value of $\bar u_1 \wedge \bar v_1 + \dots + \bar u_g \wedge \bar v_g$
and the subgroup generated by the elements $\bar u_i$ and $\bar v_i$.
This can be seen directly from the definition.
\end{remark*}

\begin{lemma} \label{lm:redution-wrt-single-vector}
Let $T = (\bar u_1,\bar v_1,\dots,\bar u_g,\bar v_g)$ be a tuple of $2g$ 
elements of $A_n$
and $\set{b_1,b_2,\dots,b_r}$ a basis of the subgroup generated by $T$.
Then $T$ can be transformed by a symplectic transformation
to a tuple $T' = (\bar u_1',\bar v_1',\dots,\bar u_g',\bar v_g')$ 
satisfying the following conditions:
$$
  u_1' = b_1 \quad\text{and}\quad \sgp{\bar v_1',\bar u_2',\bar v_2',\dots, \bar u_g', \bar v_g'} = \sgp{b_2,\dots,b_r}.
$$
\end{lemma}

\begin{proof}
Denote $\sgp{b_2,\dots,b_r}$ by $P$.
Using (S2) we can act on each pair $(\bar u_i, \bar v_i)$ as
$SL_2(\Ints)$.
In particular, we can implement the Euclidean algorithm on
the $b_1$-components of $\bar u_i$ and ~$\bar v_i$, and
transform $T$ so that $v_i \in P$ for each $i$.
Next, using (S1) and (S3) we can act as $SL_n(\Ints)$ on 
$(\bar u_1, \bar u_2, \dots, \bar u_n)$
and change $T$ so that
$$
  \bar u_1 = b_1 + h\, \mbox{ and }\, h, \bar v_1, \bar u_2, \dots \bar v_g \in P. 
$$
The elements $\bar v_1$, $\bar u_2$, $\bar v_2$, $\dots$, $\bar v_g$ 
generate $P$. In particular, we have:
$$
  h = s_1 \bar v_1 + s_2 \bar u_1 + s_3 \bar v_2 + \dots 
    + s_{2g-1} \bar v_g,
$$
for some coefficients $s_1,\ldots,s_{2g-1}$.
Using (S2) we make $s_1=0$ and then using (S3) and (S4) with $i=1$ 
we eliminate the entire ~$h$ from $\ovu_1$.
\end{proof}

\subsection{Restricting to finitely many subgroups $L$}
\label{se:finite-set-L}
We now show that we can restrict to an effective finite set of
possibilities for $L$ in \eqref{eq:constants-reduction}.
The idea can be roughly described as follows: if $L$ is 
``too stretched'' in some direction, then 
elements $\bar u_i$ and $\bar v_i$ in \eqref{eq:constants-reduction}
(and the depending subgroup $L$) can be changed making
the dimension of $L$ smaller. To control ``the stretching
factor'' we use Hermit's reduced basis for $L$ which is, 
in some sense, ``close to orthogonal'' as much as possible.

We use the standard notations $||x||$ for the norm of 
a vector $x$ in a Euclidean normed space $\Reals^r$ and 
$x \cdot y$ for the scalar product of $x$ and ~$y$.
If $E = \set{e_1,e_2,\dots,e_r}$ is a canonical orthonormal 
basis for $\Reals^r$, then we view
$\La^2(\Reals^r)$ as a Euclidean normed space
with orthonormal basis 
$E^{\wedge2} = \setof{e_i \wedge e_j}{1 \le i < j \le r}$.

\begin{remark} \label{rm:external-square-bases}
If $E' = \set{e_1', e_2',\dots,e_r'}$ is another 
orthonormal basis for ~$\Reals^r$ then $E'^{\wedge2}$
is an orthonormal basis for $\La^2 (\Reals^r)$. This can be seen
using the fact that the orthogonal group $O(r)$ is generated
by rotations in the coordinate planes $\CS(e_i,e_j)$ and 
a single reflection $e_1 \mapsto -e_1$. Then checking 
orthonormality of $E'^{\wedge2}$ is reduced to the case 
of dimension $r \le 4$ which can be done by direct computation.
\end{remark}

\begin{proposition}[Hermite's reduced basis, see for example {\cite[Chapter~2]{NV}}] \label{lm:quasi-orthogonality}
Let $L$ be a discrete subgroup of ~$\Reals^n$, $\dim_\Ints (L) = r$. Then there exists a basis $(b_1,\dots,b_r)$ for $L$
satisfying the following. Define
$$
  L_i = \sgp{b_{i+1},b_{i+2},\dots,b_r}, \any i=0,\dots,r.
$$
Let $b_1^*$, $b_2^*$, $\dots$, $b_r^*$ be obtained from $b_1$, $b_2$,
$\dots$, $b_r$ 
by the Gram--Schmidt orthogonalization process starting from $b_r$;
that is, $b_i^*$ is the orthogonal projection of $b_i$ to
the orthogonal complement of $L_i$.
Then
\begin{equation} \label{eq:reduced-basis1}
  ||b_i^*|| \ge \frac{\sqrt3}2\, ||b_{i+1}^*||, \any i=1,\dots,r-1,
\end{equation}
and $|b_i \cdot b_j^*| \le \frac12 ||b_j^*||^2$ for $i<j$ or,
equivalently,
\begin{equation} \label{eq:reduced-basis2}
||\proj_{b_j^\ast}(b_i)|| \le \frac{||b_j^\ast||}{2}.
\end{equation}
\end{proposition}

If $V$ is a nontrivial discrete subgroup of $\Reals^n$, then by
$\Vol(V)$ we denote the Euclidean volume of the parallelepiped spanned by
a free abelian basis for $V$. 
Recall that $\Vol(V)^2 = \det(AA^T)\ge 1$, where $A$
is the $r\times n$ matrix of the generators of $V$.
Therefore, $\Vol(V) \ge 1$ for any non-trivial $V \seq \Ints^n$. 


Suppose that $(b_1,\dots,b_r)$ is a basis satisfying the conclusion of 
Proposition \ref{lm:quasi-orthogonality}. Since 
$\Vol(L) = ||b_1^*|| ||b_2^*|| \dots ||b_r^*||$, it follows from
\eqref{eq:reduced-basis1} that
\begin{equation} \label{eq:basis-prop1}
  ||b_1^*|| \ge \om^{\frac{r-1}2} (\Vol(L))^{\frac 1r},
\mbox{ where } \om = \frac{\sqrt3}2.
\end{equation}
Next, from \eqref{eq:reduced-basis2} for 
$j=i+1,\dots,r$ we obtain
\begin{align*}
  ||b_i||^2 &\le \sum_{j=i}^r ||\proj_{b_j^\ast}(b_i)||^2\\
    &\le ||b_i^*||^2 + \frac14 ||b_{i+1}^*||^2 + \dots 
    + \frac14 ||b_r^*||^2 \\
    &\le ||b_i^*||^2 \left(1 + \frac1{4\om^2} + \frac1{4\om^4} + \dots + \frac{1}{4\om^{2r-2i}} \right) \\
    & < 2 \om^{2i-2r} ||b_i^*||^2.
\end{align*}
Therefore,
\begin{equation} \label{eq:basis-prop2}
    ||b_i|| < 2 \om^{i-r} ||b_i^*|| <  2 \om^{1-r} ||b_1^*||, 
    \any i=1,\dots,r.
\end{equation}
In particular,
\begin{equation} \label{eq:basis-prop3}
  \sin \angle(b_i,L_i) = \frac{||b_i^*||}{||b_i||} > \frac12 \om^{r-i},
  \any i=1,\dots,r,
\end{equation}
where $\angle(x,y)$ denotes the Euclidean angle between $x$ and $y$.

In the case when $L$ is a subgroup of $\Ints^n$, we have $||b_i|| \ge 1$
for each ~$i$ and, hence, $||b_i^*|| > \frac12 \om^{r-1}$ 
by the first inequality \eqref{eq:basis-prop2}. 
Then from 
$\Vol(L) = ||b_1^*|| ||b_2^*|| \dots ||b_r^*||$ 
and $||b_i|| < 2 \om^{i-r} ||b_i^*||$ 
we get an upper bound on $||b_i||$:
\begin{equation} \label{eq:basis-prop4}
    ||b_i|| < 2^r \om^{-\frac{r(r-1)}2} \Vol(L),  \any i=1,\ldots,r.
\end{equation}

\begin{proposition} \label{pr:L-bound}
Assume that system \eqref{eq:constants-reduction} has a solution.
Then \eqref{eq:constants-reduction} has a solution
$(\ovu_1,\ovv_1,\ldots,\ovu_g,\ovv_g)$ 
such that the subgroup
$$
L = \gp{\ovu_1,\ovv_1,\ldots,\ovu_g,\ovv_g,\ovc_1,\ldots,\ovc_m}
$$
has a basis $\set{f_1,f_2,\dots,f_k}$ satisfying
\begin{equation} \label{eq:L-basis-bound}
    ||f_i|| \le n 2^n \om^{-\frac{n(n-1)}2} D^n 
        + 2 \om^{1-n} ||h|| \left( \max_i ||\bar c_i|| \right)^n,
    \any i=1,\ldots,k,
\end{equation}
where 
$
    D = \max_i ||\bar c_i|| + \diam\supp(\de).
$
\end{proposition}

\begin{proof}
Consider a solution 
$(\bar u_1, \bar v_1, \dots, \bar u_g, \bar v_g)$
of \eqref{eq:constants-reduction}. We will show that
if $L$ is, in a certain sense, ``sufficiently large''
then the solution $(\bar u_1, \bar v_1, \dots, \bar u_g, \bar v_g)$
can be changed to decrease the dimension of $L$.
As above, let
$
    Q = \sgp{\bar c_1, \bar c_2, \dots, \bar c_m}.
$ 
We proceed in two steps.

\emph{Step 1.}
Define a (finite) subset $\De$ of $A_n$:
$$
    \Delta = \setof{g \in A_n}{gs = t \text{ for some } s,t\in \supp(\de)}.
$$
Let $L_0$ be the minimal direct summand of $L$ 
containing both $Q$ and $\De\cap L$.
Since $\pi_{L}(\de) = 0$, the equality 
$\pi_{L_0}(\de) = 0$ holds by construction.
Let $U$ be a maximal set of linearly independent 
elements of 
$\set{\bar c_1, \bar c_2, \dots, \bar c_m} \cup (\De \cap L)$. 
Then $\sgp{U}$ has finite index in 
$\sgp{Q \cup(\De\cap L)}$ and, hence, in $L_0$. 
%
Since $||u||\le D$ for each $u\in U$, we have
$$
    \Vol(L_0) \le \Vol(\sgp{U}) 
    \le D^n.
$$
By Proposition \ref{lm:quasi-orthogonality} 
and \eqref{eq:basis-prop4} there exists a basis 
$\set{f_1,f_2,\dots,f_q}$ for ~$L_0$ satisfying
\begin{equation} \label{eq:L0-basis-bound}
    ||f_i|| \le 2^n \om^{-\frac{n(n-1)}2} D^n, \any i=1,\ldots,q.
\end{equation}

\emph{Step 2.}
Here we perform the reduction step if $L$ is ``sufficiently large''.
We consider the free abelian group $A_n$ with 
basis $(\bar a_1, \bar a_2, \dots, \bar a_n)$
as canonically embedded in $\Reals^n$. 
For $X \seq \Reals^n$, denote by $\CS(X)$ the linear $\Reals$-subspace 
spanned by $X$.
Let $L_0^\perp$ be the orthogonal complement of $\CS(L_0)$ 
and let
$\pi: \Reals^n \to L_0^\perp$ denote the orthogonal projection.

Choose a basis $(b_1,\dots,b_r)$ for $\pi(L)$ 
as in Proposition \ref{lm:quasi-orthogonality}
and let vectors $b_1^*$, $b_2^*$, $\dots$, $b_r^*$ be
obtained from $b_1, b_2,\ldots, b_r$ 
by the Gram--Schmidt orthogonalization process starting from $b_r$.
By \eqref{eq:basis-prop2},
$$
    ||b_i|| < 2 \om^{1-r} ||b_1^*||, \any i=1,\dots,r.
$$
For each $i$, choose $d_i \in L$ such that $b_i = \pi(d_i)$. 
We may assume that $d_i - b_i$ is a real linear 
combination of vectors in the basis $\set{f_1,f_2,\dots,f_q}$ for $L_0$
with coefficients in the interval $[0,1)$
(otherwise we add an appropriate integral linear combination of
$f_1,\ldots,f_q$ to $d_i$). Hence,
\begin{equation} \label{eq:L-extra-basis-bound}
    ||d_i|| < 2 \om^{1-r} ||b_1^*|| + \sum_{i=1}^q ||f_i||, 
    \any i=1,\dots,r.
\end{equation}
By construction, the tuple $(d_1,\dots,d_r)$ is a basis of a free 
abelian subgroup of $L$ and 
$L = L_0 \oplus \sgp{d_1,\dots,d_r}$. 
We prove that if $||b_1^*||$ 
is sufficiently large, then there is a solution 
of \eqref{eq:constants-reduction} with the subgroup
$$
    L' = \sgp{L_0,d_2,\dots,d_r},
$$
instead of $L$.
Observe that $d_1$ and $b_1$ have the same orthogonal projection $b_1^*$
to the orthogonal complement of $\CS(L_0)$.
By Lemma \ref{lm:redution-wrt-single-vector}, we can assume that
$$
    \bar u_1 \in d_1 + L', \quad \bar v_1, \bar u_2, \bar v_2, 
        \dots \bar v_{2g} \in L'.
$$
Below we prove that the inequality
\begin{equation} \label{eq:big-L-assumption}
    ||b_1^*|| > ||h|| \Vol(Q),
\end{equation}
implies that $\bar v_1 \in Q$. 
On the way to contrary assume that $\bar v_1 \notin Q$. 
Choose an orthonormal basis $(e_1,e_2,\dots,e_n)$ for $\Reals^n$
compatible with the following chain of linear subspaces of $\Reals^n$:
$$
    \CS(Q) \seq \CS(\sgp{Q,\bar v_1}) \seq \CS(L') \seq \CS(L),
$$ 
that is,
\begin{align*}
\CS(Q) &= \CS(e_1,\dots,e_{p-1}), \\
\CS(\sgp{Q,\bar v_1}) &= \CS(e_1,\dots,e_p), \\
\CS(L') &= \CS(e_1,\dots,e_{r-1}), \\ 
\CS(L) &= \CS(e_1,\dots,e_{r}).
\end{align*}
We have
$$
    \bar u_1 \in \pm ||b_1^*|| e_r + \CS(e_1,\dots,e_{r-1})
    \ \ \mbox{ and }\ \ 
    \bar v_1 \in \al e_p + \CS(e_1,\dots,e_{p-1}),
$$
for some $\al \ne 0$.
Note that $\Vol(\sgp{Q, \bar v_1}) = |\al| \Vol(Q)$ and hence
$|\al| \ge (\Vol(Q))^{-1}$.
Project both sides of the equality
$$
  \bar u_1 \wedge \bar v_1 + \dots + \bar u_g \wedge \bar v_g
  = h \bmod K,
$$
onto the basis bivector $e_p \wedge e_r$ of $\La^2 (\Reals^n)$.
The projection of the right-hand side is bounded above 
by $||h||$ (because $K$
is projected onto 0).
The projection of each $\bar u_i \wedge \bar v_i$ for $i=2,\dots,g$ is 0
because $\bar u_i,\bar v_i \in L'$. The projection of 
$\bar u_1 \wedge \bar v_1$ is $\pm ||b_1^*|| \al $
that has the absolute value greater than ~$||h||$ 
by \eqref{eq:big-L-assumption}. 
This gives a contradiction. Thus, $\bar v_1 \in Q$.

Observe that $\bar v_1 \in Q$ implies that $L \wedge \bar v_1 \seq K$ and,
hence, the tuple
$$
    (\bar u_1 - d_1, \bar v_1, \bar u_2,\bar v_2,\dots,\bar u_g,\bar v_g),
$$
is a solution of \eqref{eq:constants-reduction} with
the subgroup ~$L'$ instead of $L$, where $\dim (L') < \dim (L)$.
Since the reduction is possible under the 
assumption \eqref{eq:big-L-assumption} we can assume that 
\eqref{eq:big-L-assumption} does not hold, i.e.:
\begin{equation} \label{eq:b1-bound}
    ||b_1^*|| \le ||h|| \Vol(Q).
\end{equation}
For the required basis $\set{f_1,f_2,\dots,f_k}$ of $L$ we then take
$$\set{f_1,f_2,\dots,f_q,d_1,d_2,\dots,d_r}.$$
Inequalities \eqref{eq:L0-basis-bound}, \eqref{eq:L-extra-basis-bound} 
and \eqref{eq:b1-bound} imply 
the required bound \eqref{eq:L-basis-bound}.
\end{proof}

%
%
%
%

Define a (finite) set $\CL$ of subgroups of $A_n$:
$$
\CL = \left\{L=\gp{f_1,\ldots,f_k} \middle| 
\begin{array}{l}
f_1,\ldots,f_k \mbox{ satisfy \eqref{eq:L-basis-bound}, } \pi_L(\delta)=0, \\
\mbox{and } c_1,\ldots,c_m \in L.
\end{array}
\right\}.
$$
Clearly, the set $\CL$ can be effectively computed for 
a given equation \eqref{eq:constants-reduction}.
The next statement is an immediate corollary of
Proposition \ref{pr:L-bound}.

\begin{corollary} \label{co:restricting-L}
A system \eqref{eq:constants-reduction} 
is solvable if and only if the following system
in unknowns $\bar u_i, \bar v_i \in A_n$ is solvable for some $L\in\CL$:
\begin{equation} \label{eq:reduction-with-L}
\begin{cases}
  \sgp{\bar u_1, \bar v_1, \bar u_2, \bar v_2, \dots, \bar u_g, \bar v_g, \bar c_1, \bar c_2, \dots, \bar c_m} = L, \\
  \bar u_1 \wedge \bar v_1 + \dots + \bar u_g \wedge \bar v_g =  h \bmod K.
\end{cases}
\end{equation}
\end{corollary}

Let $R = L/Q$. As observed above, the subgroup $K$ of $\La^2(L)$
is the kernel of the natural epimorphism $\La^2(L) \to \La^2(R)$.
Passing to images of $\bar u_i$, $\bar v_i$ in $R$
and to the image of ~$h$ in $\La^2(R)$ (and, with a slight abuse,
coming back to notations $u_i$, $v_i$ and $h$)
we rewrite the system
\eqref{eq:reduction-with-L} as follows:
\begin{equation} \label{eq:reduction-mod-Q}
\begin{cases}
  \sgp{u_1, v_1, u_2, v_2, \dots, u_g, v_g} = R, \\
  u_1 \wedge v_1 + \dots + u_g \wedge v_g =  h.
\end{cases}
\end{equation}
Thus, the Diophantine problem for the equation \eqref{eq:equation} is reduced 
to the following problem: {\em Given a finitely generated abelian group $R$,
an element $h \in \La^2 (R)$ and a number $g$, determine if there exist elements $u_i, v_i \in R$
satisfying \eqref{eq:reduction-mod-Q}.}

\section{Solution of the reduced problem}
\label{se:free_abelian_case}

Here we prove that the reduced problem is algorithmically decidable.

\begin{proposition}\label{th:finite_reduction}
There is an algorithm that determines if 
the system \eqref{eq:reduction-mod-Q}
has a solution $(u_1,v_1,\ldots,u_g,v_g)$ or not.
\end{proposition}

We prove Proposition \ref{th:finite_reduction}
for free abelian $R$ in Section \ref{se:free_case}.
The general case is considered in Section \ref{se:general_case}.

\subsection{The case of free abelian $R$}\label{se:free_case}

Suppose that $R$ is a free abelian group, equipped
with a fixed basis $E = \set{e_1,e_2,\dots,e_r}$, and naturally embedded in 
the normed space $\Reals^n$ with the Euclidean norm $||\cdot||$.
We view $\La^2 (R)$ as embedded in $\La^2 (\Reals^r)$;
on $\La^2 (\Reals^r)$ we have an induced norm defined by 
setting $E^{\wedge2} = \setof{e_i \wedge e_j}{1 \le i < j \le r}$ to be 
an orthonormal basis.
Below we find an effective bound on the sizes of elements 
$u_i$ and $v_i$ satisfying \eqref{eq:reduction-mod-Q}.

\begin{lemma} \label{lm:angle-in-Lambda2}
Let $0 \ne b \in \Reals^r$ and $L$ be a subspace 
of $\Reals^r$ with $\dim(L) \ge 2$.
Then the angle between subspaces $\La^2 (L)$ and $b \wedge L$ of 
$\La^2 (\Reals^r)$ is equal to the angle between ~$b$ and ~$L$.
\end{lemma}

\begin{proof}
If $b \in L$, then both angles are 0. Hence, we can assume that $b \notin L$
and, consequently, $\dim (L) < r$.
By Remark \ref{rm:external-square-bases}, 
we can assume that the basis $E$ is 
chosen in a ``nice'' way with respect to $b$ and $L$, e.g.:
$$
L = \CS(e_1,e_2,\dots,e_k) \quad\text{and}\quad
b = ||b|| \, (e_1 \cos \al + e_r \sin \al),
$$
where $k < r$ and $\al = \angle (b,L)$. 
Note that $(b \wedge L) \cap \La^2 (L) = 0$. If $S_1$
and $S_2$ are two subspaces of $\Reals^m$ with $S_1 \cap S_2 = 0$,
then the angle between $S_1$ and ~$S_2$ is the minimal possible 
angle between
a vector $x \in S_1$ and its orthogonal projection onto $S_2$. 
In our case, $S_1 = b \wedge L$, $S_2 = \La^2 (L)$ and
$x = b \wedge y$ for some $y \in L$. Up to further change of $E$
we can assume that $y \in \CS(e_1,e_2)$. Then an easy computation
shows that this angle is precisely ~$\alpha$.
\end{proof}

\begin{proposition} \label{pr:solution-bound}
Assume that $R$ is a free abelian group of rank ~$r$.
Then any solution $(u_1,v_1,\dots u_g,v_g)$ 
of the system \eqref{eq:reduction-mod-Q} can be transformed 
by symplectic transformations
into a new solution $(u_1^\ast, v_1^\ast,\ldots,u_g^\ast, v_g^\ast)$
such that for each $i$:
\begin{equation} \label{eq:solution-bound}
    ||u_i^*||, ||v_i^*|| < 2^{r^2} (||h|| + 1).
\end{equation}
\end{proposition}

\begin{proof}
We consider $R$ as embedded in the Euclidean space $\Reals^r$.
For any $x,y \in \Reals^r$ we have:
\begin{equation} \label{eq:norm-cap-product}
  ||x \wedge y|| = \sin \angle(x,y) \cdot ||x|| \: ||y||.
\end{equation}

{\em Step 1:\/}
By induction on $g$ we prove the following.
{\em If a tuple $(u_1,v_1,\dots u_g,v_g) \in R^{2g}$ satisfies:
$$
  u_1 \wedge v_1 + \dots + u_g \wedge v_g =  h,
$$
then it can be transformed by symplectic transformations to 
the form:
$$
  (u_1,v_1,\dots,u_k,v_k, u_{k+1}, 0, u_{k+2}, 0, \dots, u_g, 0),
$$
where $||u_i||, ||v_i|| \le 4^r \om^{-2r(r-1)} ||h||$
for each $i = 1,\ldots,k$.}

Choose a basis $B = \sgp{b_1,b_2,\dots,b_r}$ for 
$\sgp{u_1,v_1,\dots u_g,v_g}$
by Lemma ~\ref{lm:quasi-orthogonality}.
Using Lemma ~\ref{lm:redution-wrt-single-vector} we may assume that
$u_1 = b_1$ and:
$$
    \sgp{v_1,u_2,\dots,v_g} = \sgp{b_2,\dots,b_r} = L_1.
$$
If $v_1 = 0$, then using transformations (S1) we put the pair $(u_1,0)$
to the end as $(u_g,0)$ and use induction. 
Assume that $v_1 \ne 0$. Let $\al$ be the Euclidean angle
between $u_1 \wedge v_1$ and $h - u_1 \wedge v_1$.
We have:
$$
  ||u_1 \wedge v_1||, \ ||h - u_1 \wedge v_1|| 
    \le \frac{||h||}{\sin \al}.
$$
Since $u_1 \wedge v_1 \in u_1 \wedge L_1$ and 
$h - u_1 \wedge v_1 \in \La^2 (L_1)$, by Lemma \ref{lm:angle-in-Lambda2}
and inequality \eqref{eq:basis-prop3} we have:
$$
  \sin \al \ge \sin \angle (u_1 \wedge L_1, \La^2 (L_1)) 
  \ge \frac12 \om^{r-1},
$$
and, hence:
$$
  ||u_1 \wedge v_1||, \ ||h - u_1 \wedge v_1|| \le 2 \om^{1-r} ||h||,
$$
By \eqref{eq:norm-cap-product},
since $||u_1||, ||v_1|| \ge 1$, we have:
$$
    ||u_1||, ||v_1|| \le 4 \om^{2-2r} ||h||.
$$
Then we use the inductive hypothesis for 
$(u_2,v_2,\dots,u_g,v_g)$ with $h - u_1 \wedge v_1$ instead of $h$.

{\em Step 2.}
Now assume that we are given a tuple
$
  (u_1,v_1,\dots,u_g,v_g)
$
satisfying \eqref{eq:reduction-mod-Q}.
Let $S_1 = \CS(u_1,v_1,\dots,u_k,v_k)$ and let $S_2$ 
be the orthogonal complement
of $S_1$ in $R$. 
Denote $\pi_i: R \to S_i$ $(i=1,2)$ the orthogonal projection map.
Using symplectic transformations (S3) we can implement the 
action of $SL_{g-k}$ on the tuple $(u_{k+1}, u_{k+2}, \dots, u_g)$.
Hence, by Lemma ~\ref{lm:quasi-orthogonality}, we can change
$u_i$ so that for some $t\le g$ the projections $\pi_2(u_i)$
$(i=k+1,\dots,t)$ form Hermit's reduced basis for $S_2$
and $\pi_2(u_i) = 0$ for $i > t$.
Let $u_i^*$ $(i=k+1,\dots,t)$ denote the vector obtained 
from $\pi_2(u_i)$ by the Gram--Schmidt orthogonalization 
process starting from $\pi_2(u_t)$. 
Since $u_1,v_1,\dots,u_g,v_g$ generate $R$ we have:
$$
    ||u_{k+1}^*|| 
        = \frac{\Vol\sgp{S_1,u_{k+1},\dots,u_t}}
        {\Vol\sgp{S_1,u_{k+2},\dots,u_t}}
        = \frac{1}
        {\Vol\sgp{S_1,u_{k+2},\dots,u_t}}
    \le 1.
$$
Then by \eqref{eq:basis-prop2}:
$$
    ||\pi_2(u_i)|| < 2 \om^{1-r}, \quad i=k+1,\dots,g.
$$
It remains to bound 
the other projections $\pi_1(u_i)$ for $i\ge k+1$.
Using transformations (S3) we can add to $u_i$ for $i\ge k+1$
any integer linear combinations of vectors $u_i$ with $i \le k$.
Hence, we can change each $u_i$ for $i \ge k+1$ so that:
$$
    ||\pi_1(u_i)|| < \sum_{i=1}^k ||u_i||
    < r\, 4^r \om^{-2r(r-1)} ||h||,
    \quad i=k+1,\dots,g.
$$
It can be easily checked that the bound obtained 
for $||u_i||$ and $||v_i||$ is less than the right-hand 
side of \eqref{eq:solution-bound}.
\end{proof}

Proposition \ref{pr:solution-bound} immediately implies
Proposition \ref{th:finite_reduction} for a free abeian group $R$.

\subsection{The case of a general abelian $R$}
\label{se:general_case}



Assume that $R$ is a general finitely generated abelian group.
To fix a way of presenting the input data, we
assume that the quotient $R = L/Q$ 
is given by specifying subgroups $L$ and $Q$ of $A_n$,
i.e.\ we are given a set $\set{c_1, c_2, \dots, c_m}$
of generators for $Q$ and a set 
$\set{e_1,e_2,\dots,e_s}$ of generators for $L$ over $Q$.
The element $h \in \La^2(R)$ 
is given by a tuple $(h_{ij})_{1\le i <j \le s}$ of (non-uniquely) defined 
integers so that $h$ is the canonical image of the linear combination $\sum_{i,j} h_{ij} (\bar a_i\wedge \bar a_j)$.
The algorithm goes as follows.

{\em Step\/} 1. 
Represent $R$ in the form
$$
    R = R_{free}\oplus R_{tor},
$$
where $R_{free}\simeq \MZ^r$
and $R_{tor}$ is a finite abelian group. 
From the computational point of view, this means that we find a 
generating set $\set{b_1,b_2,\dots,b_q}$ of $R$ such that 
$B = (b_1,b_2,\dots,b_r)$ is a basis for $R_{free}$ and 
$R_{tor}$ is the direct
sum of finite cyclic subgroups $\sgp{b_i}$ of orders $d_i$
for $r+1 \le i \le q$.

{\em Step\/} 2. 
By Proposition \ref{pr:solution-bound} any solution 
$(u_1,v_1,\dots,u_g,v_g)$ of \eqref{eq:reduction-mod-Q}
can be modified in such a way that its coefficient matrix 
$T = (t_{ij})_{ij}$ in the generators $b_j$
satisfies:
$$
    |t_{ij}| \le  2^{r^2} (||h||_{B^{\wedge2}} + 1) \quad\text{for }
    j \le r,
$$
where $||h||_{B^{\wedge2}}$ is the norm of $h$ written in 
the basis $B^{\wedge2} = \setof{b_i \wedge b_j}{1 \le i < j \le r}$.
Since $\sgp{b_j}$ has finite order $d_j$ for $j > r$, we obtain also:
$$
    |t_{ij}| < d_j \quad\text{for } j > r.
$$
This gives a finite search space for a solution 
of \eqref{eq:reduction-mod-Q}. Checking each equality in 
\eqref{eq:reduction-mod-Q} is obviously algorithmic.
This finishes the proof of Proposition \ref{th:finite_reduction}
and Theorem \ref{th:main_theorem}.

%
%
%

\section{Complexity analysis}
\label{se:complexity}


According to the main result of ~\cite{Lysenok_Ushakov:2016},
the Diophantine problem for orientable quadratic equations
is $\NP$-hard.
In this section we prove the upper bound for 
Theorem \ref{th:main_theorem2} and Theorem \ref{th:main_theorem3}:
the problem belongs to $\NP$ and if 
the rank $n$ of the free metabelian group $M_n$ and 
the number of variables in the equation are fixed, 
then there exists a polynomial-time algorithm.

We rely on existence of polynomial-time algorithms \cite{KB79} solving 
two principal problems of integer linear algebra:
reduction of an integer matrix to the Hermite and the Smith normal forms.
More precisely, there exist the following algorithms working in 
time bounded by a polynomial on the size of the input
(see \cite{S00}):

\medskip\noindent
\textbf{Extended Hermite normal form algorithm:}
Given an integer $(p\times q)$-matrix $M$, the algorithm 
computes the Hermite normal form $M^\dag$ of ~$M$ 
and a matrix $U \in GL(p,\MZ)$ such that $M^\dag = UM$.
Absolute values of entries of $U$ and $M^\dag$ are bounded by 
$r^r ||M||^r$, where $r = \max(p,q)$ and $||M||$ is the largest
absolute value of an entry in $M$.

\medskip\noindent
\textbf{Extended Smith normal form algorithm:}
Given an integer $(p\times q)$-matrix $M$, the algorithm 
computes the Smith normal form $M^\ddag$ of $M$ and matrices 
$U \in GL(p,\MZ)$ and $V \in GL(q,\MZ)$ 
such that $M^\ddag = UMV$.
Absolute values of entries of $U$ and $V$ are bounded by
$r^{4r+5} ||M||^{4r+1}$. 

\medskip
As a consequence, we obtain a polynomial-time algorithm for
the following version of the membership 
problem for finitely generated free abelian groups:
given a finite set $S$ of elements of a finitely generated free abelian
group $A$ and an element $x \in A$, determine if $x$ belongs to 
the subgroup $\sgp{S}$. In particular, given two elements $x,y \in A$
we can determine in polynomial time
whether $x$ and $y$ belong to the same coset modulo $S$.
It is assumed here that elements of $A$ are presented as
vectors of integers with entries written in the binary form.

Note that the diagonal entries of 
the Smith normal form of a matrix ~$M$ are equal to the
greatest common divisor of certain sets of minors of ~$M$. 
In particular, if $d$ is such an entry, then by Hadamard's
inequality we have $|d| \le r^{r/2} ||M||^r$.

The input data for our problem
consists of the number $g$ and 
the coefficients $c_1$, $\dots$, $c_m$ written as words in 
the alphabet $\set{a_1^{\pm1}, \dots, a_n^{\pm1}}$.
As observed in Introduction, we can assume that $g \le n$, i.e.\ if
the group $M_n$ is fixed then $g$ is bounded by a constant. 
For the size of 
the input we take $N = \sum_{i=1}^m |c_i| + m + n$.

We analyze computational complexity of each step of the algorithm
and prove existence of an appropriate $\NP$-certificate.
For fixed ~$m$ and ~$n$ we show that the step can be performed 
in polynomial time. By writing $x \le O(F(N,m,n))$ we mean that
$x \le c\cdot F(N,m,n)$ for some constant $c$ 
independent of $N$, $m$ and $n$.

{\em Restricting to finitely many $w_i$'s.}
Proposition \ref{pr:bounding-w} reduces the problem to checking 
existence of a solution of finitely 
many systems \eqref{eq:coefficient-condition} $\&$ 
\eqref{eq:reduced-homology-condition} over all possible 
choices of tuples 
$\ov W = (\bar w_1, \bar w_2, \dots, \bar w_m) \in (A_n)^m$
satisfying $|\bar w_i| \le \sum_j |c_j|$.
The record size of $\ov W$ is bounded by $O(N^2)$,
so we can consider $\ov W$ as a part of an $\NP$-certificate for solvability
of the initial system \eqref{eq:equation}
and, thus, reduce the problem to a single choice of ~$\ov W$.
The size of the search space for $\ov W$ is bounded by $O((2N+1)^{mn})$
which is polynomial in $N$ when $m$ and $n$ are fixed. 

For each $\ov W$ we reduce the problem to a system 
of the form \eqref{eq:constants-reduction}. 
It can be easily seen from the proof 
of Corollary ~\ref{co:constants-reduction} that
for constants $\de  \in C_1(\Ga_n)$ and 
$h = \sum_{1 \le i < j \le n} h_{ij} (\bar a_i \wedge \bar a_j)$ 
occurring in \eqref{eq:constants-reduction} we have:
\begin{equation} \label{eq:delta-and-h-bounds}
    \diam\supp(\de) \le 2N\, \mbox{ and }\,  \max_{i,j} |h_{ij}| \le 2 N^2.
\end{equation}

{\em Restricting to finitely many subgroups $L$.}
The next step reduces the system \eqref{eq:constants-reduction} 
to finitely many systems \eqref{eq:reduction-with-L}. 
Each system \eqref{eq:reduction-with-L} is defined 
by a choice of a subgroup $L$ generated by a basis 
$U = \{f_1,\ldots,f_k\}$ satisfying inequality \eqref{eq:L-basis-bound}
and the equality $\pi_{L}(\de)=0$. 
The right-hand side of \eqref{eq:L-basis-bound}
is bounded by $O(C^{n^2} N^{n+1})$ for some constant $C > 1$.
Hence, the number of possible choices for $U$ is bounded by 
$O(C^{n^3} N^{n(n+1)})$.
For each choice of $U$, the equality $\pi_{L}(\de)=0$
and the membership $c_i\in L$
can be checked in polynomial time using the algorithm 
for the membership problem in finitely generated free abelian groups.
Thus, we can include a generating set $U$ for $L$
into our $\NP$-certificate and reduce the problem to 
a single choice of $L$. 
For a fixed $n$ the set $\CL$ can be computed in polynomial time.

As a final step of the reduction, 
for a particular choice of a generating set for $L$
we rewrite the system \eqref{eq:reduction-with-L} 
in the form ~\eqref{eq:reduction-mod-Q}. 
The input data for ~\eqref{eq:reduction-mod-Q} is already 
presented in the form described in Section \ref{se:general_case}.



{\em Solution of the reduced problem.}
We follow the procedure described in Section \ref{se:general_case}
which consists of two steps: at Step 1 we obtain
a normalized basis for the abelian group $R = L/Q$
and at Step 2 we give a bound on the solution in terms of the new basis.
We show that Step~1 can be performed in polynomial time
and bound at Step 2 gives a solution of the size 
bounded by a polynomial in $N$.

The input data for ~\eqref{eq:reduction-mod-Q} consists of three matrices: 
a matrix $K_0$ expressing the generators for $Q$ in 
the basis $B_0 = \set{\bar a_1, \bar a_2, \dots, \bar a_n}$,
a matrix ~$M_0$ expressing the generators for $L$ in $B_0$
and a skew-symmetric matrix ~$H_0$ expressing $h$
in the basis $B_0^{\wedge 2}$ of $\La^2(A_n)$.
We have upper bounds:
$$
    ||K_0|| \le N, \quad ||M_0|| \le C^{n^2} N^{n+1}, \quad ||H_0|| \le 2N^2.
$$
The computation consists of the following steps.

{\em Transferring the data to a basis of $L$.}
We reduce the matrix $M_0$
to a Hermite normal form $M_1$ producing a basis $B_1$ for $L$. 
According to the above bound we can assume that:
$$
    ||M_1|| \le n^n ||M_0||^n.
$$
We express the generators for $Q$ in the new basis $B_1$. For the 
new matrix $K_1$ of coordinates of generators of $Q$ we have:
$$
    K_1 = K_0^* \, {M_1^*}^{-1},
$$
where $K_0^*$ and $M_1^*$ are certain submatrices of $K_0$ and $M_1$.
This gives the bound:
$$
    ||K_1|| \le n ||K_0|| \cdot ||{M_1^*}^{-1}|| 
    < n ||K_0|| \cdot n^{n/2} ||M_1||^n.
$$
In a similar way, for the new matrix $H_1$ of coordinates of $h$ in the
basis $B_1^{\wedge2}$ we have:
$$
    H_1 = \bigl({M_1^{*}}^{-1} \bigr)^T H_0^* \, {M_1^*}^{-1},
$$
and, hence:
$$
    ||H_1|| < n^2 ||H_0|| \cdot n^n ||M_1||^{2n}.
$$

{\em Representing the subgroup $R = L/Q$ as
$
    R_{free}\oplus R_{tor}
$.}
This requires reducing the matrix $K_1$ to the Smith normal form $K_2$.
We have
$$
    K_2 = U K_1 V
$$
for some unimodular matrices $U$ and $V$ satisfying
$$
    ||U||, ||V|| \le n^{4n+5} ||K_1||^{4n+1}.
$$
We obtain a basis $B_2 = \set{b_1,b_2,\dots,b_q}$ of $R$ such that 
$R_{free} = \sgp{b_1,b_2,\dots,b_r}$ and 
$R_{tor} = \MZ_{d_{r+1}} \oplus \dots \oplus \MZ_{d_q}$ 
where $\MZ_{d_i} = \sgp{b_i}$ for $i > r$.
The numbers ~$d_i$ are non-unit diagonal elements of $K_2$ and
we have
$$
    d_i < n^{n/2} \, ||K_1||^n, \quad i=r+1,\dots,q.
$$
We express the canonical image of $h$ in the new basis $B_2^{\wedge2}$
obtaining a matrix $H_2$ computed as
$$
    H_2 = \bigl({V^{*}}^{-1} \bigr)^T H_1^* \, {V^*}^{-1}
$$
This gives the bound
$$
    ||H_2|| < n^2 ||H_1|| \cdot n^{n} ||V||^{2n}.
$$

{\em The final solution bound.}
According to Step 2 of the procedure in \ref{se:general_case},
if the system  ~\eqref{eq:reduction-mod-Q} is solvable then it 
has a solution  $(u_1,v_1,\dots,u_g,v_g)$ represented by a matrix $T$
in the basis $B_2$ satisfying
$$
    ||T|| \le  \max\bset{2^{n^2} (||H_2|| + 1), \ \max_i |d_i| }
$$
We observe that all numerical upper bounds in our computation 
are of the form $O(2^{f(n)} N^{g(n)})$ where $g(n)$ and $f(n)$ 
are bounded by polynomials on $n$. This gives a similar bound on $||T||$.
This implies that the record size of $T$ is bounded by a polynomial
on the input size of the problem, and hence the problem belongs to $\NP$.
If $n$ is fixed, we obtain a search space for $T$ of a polynomial size.
The proof is completed.

\section{Open problems}

We left the case of non-orientable quadratic equations open.
We think that methods of the current paper can be appropriately modified
to treat that case as well.

As mentioned in Introduction, the general Diophantine problem in free metabelian groups is undecidable. It would be interesting to
investigate some specific (non-quadratic) classes of equations. 
For instance, Baumslag, Mahler, and Lyndon in
\cite{Baumslag-Mahler:1965,Lyndon:1966} 
study equations of the form $x^ny^mz^l=1$. 
Also, Roman'kov in \cite{Roman'kov-2017} investigates solvability of regular
one-variable equations in the class of metabelian groups.

Spherical equations are a straightforward generalization of
conjugacy equations. 
It is shown in \cite{Noskov:1982} that the conjugacy problem 
in finitely generated metabelian groups is decidable.
It is currently unknown if a similar result holds
for spherical quadratic equations. 
It is also interesting to investigate the Diophantine problem 
for quadratic equations in the whole class of finitely generated metabelian groups
and in specific groups such as Baumslag--Solitar groups. 

Another possible generalization of our work is a study of groups
of type $F/N'$ where $F$ is a free group and 
$N \trianglelefteq F$. Recall that the power problem in a group $G$
is to determine if $u\in\gp{v}$ for given $u,v\in G$.
It is shown in \cite{Gul-Sohrabi-Ushakov:2017} 
that the conjugacy problem in $F/N'$
is decidable if and only if 
the power problem in $F/N$ is decidable. 
It would be interesting to 
generalize that result at least to spherical quadratic equations.



\begin{thebibliography}{99}


\bibitem{MalRhe-1998}
M. Akhavan-Malayeri and A.H. Rhemtulla,
\emph{Commutator length of abelian-by-nilpotent groups},
Glasgow Math. J. \textbf{40} (1998), 117--121.


\bibitem{Baumslag-Mahler:1965}
G.~{Baumslag}, K.~{Mahler}, \emph{Equations in free metabelian groups},
  Michigan Math. J. \textbf{12} (1965), 417--420.


\bibitem{CE} L. Comerford, C. Edmunds,
{\it Quadratic equations over free groups and free products},
J. Algebra, 68, 1981, 2, 276--297.
 

\bibitem{Diekert-Jez-Plandowski:2016}
V.~Diekert, A.~Jez, W.~Plandowski,
\emph{Finding All Solutions of Equations in Free Groups and 
Monoids with Involution}, Information and Computation,
\textbf{251} (2016), 263--286.


\bibitem{DLS}
C.~{Droms}, J.~{Lewin}, H.~{Servatius}, \emph{{The length of elements in
  free solvable groups}}, Proc. Amer. Math. Soc. \textbf{119} (1993), 27--33.
  




\bibitem{GK}
    R. Grigorchuk, P. Kurchanov,
{\em Some questions of group theory related to geometry},
A.~N.~Parshin, I.~R,~Shafarevich (Eds.), Algebra VII. Combinatorial group theory.
Applications to geometry. Springer 1993. 167--232. 

\bibitem{Gul-Sohrabi-Ushakov:2017}
F.~Gul, M.~Sohrabi, A.~Ushakov, 
{\em Magnus embedding and algorithmic properties of groups $F/N^{(d)}$}. 
Transactions of American Mathematical Society, \textbf{369} (2017), 6189--6206.

\bibitem{FM}
B.~{Farb}, D.~{Margalit}, \emph{{A primer on mapping class groups}}.
  Princeton University Press, 2011.

\bibitem{KB79}
R.~Kannan, A.~Bachem, {\em Polynomial algorithms for computing the
Smith and Hermite normal forms of an integer matrix}, SIAM J.\ Comput.\ 
\textbf{8} (1979) 4, 499--507.

\bibitem{KLMT} 
O. Kharlampovich, I. G. Lys\"enok, A. G. Myasnikov, N. W. M. Touikan,
{\it The Solvability Problem for Quadratic Equations over Free Groups is NP-Complete,}
Theory Comput. Syst.  47  (2010),  no. 1, 250--258.
  
\bibitem{Kharlampovich-Mohajeri-Taam-Vdovina:2017} 
O. Kharlampovich, A. Mohajeri, A. Taam, A. Vdovina,
{\it Quadratic equations in hyperbolic groups are $\NP$-complete,}
Trans. Amer. Math. Soc., \textbf{369} (2017),  6207--6238.




\bibitem{Lyndon:1966}
R.~{Lyndon}, \emph{{Equations in free metabelian groups}}, Proc. Amer. Math.
  Soc. \textbf{17} (1966), 728--730.

\bibitem{Lysenok_Miasnikov_Ushakov:2016}
I. Lysenok, A. Miasnikov, A. Ushakov, 
{\em Quadratic equations in the Grigorchuk group}. 
Groups Geom. Dyn., \textbf{10} (2016), 201--239.


\bibitem{Lysenok_Ushakov:2016}
I. Lysenok, A. Ushakov, {\em Spherical quadratic equations in free metabelian groups}. Proc. Amer. Math. Soc., \textbf{144} (2016), 1383--1390.



\bibitem{Mak} G. Makanin, {\it Equations in a free group},
  Izv. Akad. Nauk SSSR Ser. Mat.,  46, 1982, 6, 1199--1273.




\bibitem{Miasnikov_Romankov_Ushakov_Vershik:2010}
A.~G. {Miasnikov}, V.~{Romankov}, A.~{Ushakov}, A.~{Vershik}, \emph{The
  word and geodesic problems in free solvable groups}, Trans. Amer. Math. Soc.
  \textbf{362} (2010), 4655--4682.

\bibitem{NV}
Ph. Q. Nguyen, B. Vall\'ee (Editors), 
The LLL Algorithm. Survey and Applications. Springer 2010.


\bibitem{Noskov:1982}
G. Noskov, \emph{Conjugacy problem in metabelian groups},
Math. Notes, \textbf{31} (1982), 252--258.





\bibitem{Rhemtulla-1969}
A.~H.~Rhemtulla, 
\emph{Commutators of certain finitely generated soluble groups}, 
Canad. J. Math., 
\textbf{21} (1969), 1160--1164

\bibitem{Roman'kov_1979}
V.~{Roman'kov}, \emph{{Equations in free metabelian groups}}, Sib. Math. J.
\textbf{20} (1979), 469--471.

\bibitem{Roman'kov:2012}
\bysame, \emph{{Equations over groups}}, Groups, Complexity, Cryptology
  \textbf{4} (2012), 191--239.

\bibitem{Roman'kov_2016}
\bysame 
\emph{Diophantine questions in the class of finitely generated nilpotent groups}, J. Group Theory,
\textbf{19} (2016), 497--514.

\bibitem{Roman'kov-2017}
\bysame 
\emph{On solvability of regular equations in the variety of metabelian groups}, 
Prikl. Diskr. Mat., 2017, 51--58.

\bibitem{Seifert-Threllfal} 
H.~Seifert, W.~Threllfal, {\em A textbook of topology}, Academic Press, 1980.


\bibitem{S00}
A.~Storjohann, \emph{Algorithms for Matrix Canonical Forms}, 
Ph.D. Thesis, ETH Zurich, 2000.






\bibitem{Vershik_Dobrynin:2004}
A.~M. {Vershik}, S.~{Dobrynin}, \emph{{Geometrical approach to the free
  sovable groups}}, Int. J. Algebra Comput. \textbf{15} (2005), 1243--1260.


\end{thebibliography}
\end{document}